\documentclass{amsart}
\usepackage{amssymb}
\usepackage{algorithmic} 
\usepackage{algorithm}
\floatname{algorithm}{Algorithm}

\usepackage{hyperref}

\newtheorem{theorem}{Theorem}[section]
\newtheorem{lemma}[theorem]{Lemma}
\newtheorem{remark}[theorem]{Remark}
\newtheorem{definition}[theorem]{Definition}

\newtheorem{corollary}[theorem]{Corollary}
\newtheorem{proposition}[theorem]{Proposition}
\newtheorem{lem-def}[theorem]{Lemma-Definition}
\newcommand{\R}{\mathbb R}

\newcommand{\Z}{\mathbb Z}
\newcommand{\Q}{\mathbb Q}

\newcommand{\F}{\mathbb F}

\def\op{\operatorname}

\def\as#1{\renewcommand\arraystretch{#1}}

\def\bb{{\mathcal B}}

\def\dsc{\operatorname{Disc}}
\def\diso{\lower.4ex\hbox{$\downarrow$}\raise.4ex\hbox{\mbox{\scriptsize
$\wr$}}}

\def\gen#1{\big\langle\, {#1} \,\big\rangle}
\def\gl#1#2{\op{GL}_{#1}(#2)}

\def\ind{\op{ind}}
\def\iso{\ \lower.3ex\hbox{\as{.08}$\begin{array}{c}\lra\\\mbox{\tiny $\sim\,$}\end{array}$}\ }

\def\lg{l\raise.6ex\hbox to.2em{\hss.\hss}l}
\def\ll{\mathcal{L}}

\def\lra{\longrightarrow}
\def\m{{\mathfrak m}}

\def\oo{\mathcal{O}}
\def\orb{\hbox to  .3em{$\backslash$}\backslash}

\def\res{\operatorname{Res}}

\def\dd{\mathcal{D}}
\def\v#1{\Vert#1\Vert}

\def\pp{\mathbb{P}}

\def\bb{{\mathcal B}}

\def\dd{{\mathcal D}}

\def\ii{\mathcal{I}}
\def\cf{{C_f}}
\def\ind{\mathrm{ind}}
\def\indi{\mathrm{ind}_\infty}
\newcounter{cs}
\stepcounter{cs}
\newcommand{\casos}{\begin{itemize}}
\newcommand{\fcasos}{\end{itemize}\setcounter{cs}{1}}

\newfont{\tit}{cmr12 scaled \magstep3}

\title[Computation of the genus]{Genus Computation of global function fields}

\author[Bauch]{Jens-Dietrich Bauch}
\address{Departament de Matem\`{a}tiques,
         Universitat Aut\`{o}noma de Barcelona,
         Edifici C, E-08193 Bellaterra, Barcelona, Catalonia, Spain}
\email{bauch@mat.uab.cat}
\thanks{Supported by  MTM2009-10359 from the
Spanish MEC}
\date{}

\makeatletter
\@namedef{subjclassname@2010}{%
  \textup{2012} Mathematics Subject Classification. \emph{Primary 11G30; Secondary 14H05, 11Y40}}

\subjclass[2010]{ }
\begin{document}

\begin{abstract}
In this paper we present an algorithm that computes the genus of a global function field. Let $F/k$ be function field over a field $k$, and let $k_0$ be the full constant field of $F/k$. By using lattices over subrings of $F$, we can express the genus $g$ of $F$ in terms of $[k_0:k]$ and the indices of certain orders of the finite and infinite maximal orders of $F$. If $k$ is a finite field, the Montes algorithm computes the latter indices as a by-product. This leads us to a fast computation of the genus of global function fields. Our algorithm does not require the computation of any basis, neither of finite nor infinite maximal order.

\end{abstract}

\maketitle
\section*{Introduction}
Let $F/k$ be a function field of one variable over $k$ and denote $\dd_F$ the set of divisors of $F$. The computation of the non-negative integer 
$$
g:=\max\{\deg_k A-\dim_k A+1\mid A\in \dd_F\}
$$
is one of the fundamental tasks in algebraic function field theory or the theory of algebraic curves. To this day, the fastest algorithms that compute the genus $g$ of $F$ are based on the computation of certain Riemann-Roch spaces \cite{F.H.}. For this purpose the computation of bases of the finite and infinite maximal orders of $F$ is necessary. If $k$ is equal to the full constant field $k_0$ of $F$, we present in this paper a direct way to determine $g$. For instance, no basis computation will be required at all. \\
Our algorithm is based on the repeated application of the Montes algorithm. Hence, it has an excellent practical performance for global function fields; that is, when $k$ is a finite field. According to our tests, the running time of the genus computation is in most of the cases dominated by the computation and factorization of the discriminant of a defining polynomial of $F$. The complexity estimation for the Montes algorithm in \cite{BNS} affords us concrete bounds for the number of operations in the finite constant field $k$, which are needed to compute the genus of $F$ (Theorem \ref{complexi}). Unfortunately, these theoretical bounds do not fit well with the practical performance of the method. 

\section{Algebraic function fields}
Throughout this paper $F/k$ will denote an algebraic function field of one variable over the field $k$. That is, $F/k(t)$ is a separable extension of finite degree $n$, for $t\in F$ transcendental over $k$. We denote $A:=k[t]$, $K:=k(t)\subset F$. Let $v_\infty:K\rightarrow\Z\cup\{\infty\} $ be the discrete valuation determined by 
$$v_\infty(h/g):=\deg g-\deg h$$
for $h,g\in A$. Let $A_\infty=k[t^{-1}]_{(t^{-1})}\subset K$ be the valuation ring of $v_\infty$, $\m_\infty$ its maximal ideal and $U_\infty:=\{a\in K\mid v_\infty(a)=0\}$, the group of units of $A_\infty$.
Denote by $\mathbb{P}_F$ the set of all places of $F/k$ and let $\mathbb{P}_\infty\subset\mathbb{P}_F$ be the set of all places over $\infty$. We set $\mathbb{P}_0:=\mathbb{P}_F\setminus \mathbb{P}_\infty$. Every place $P\in\pp_F$ corresponds to a surjective valuation $v_P:F\rightarrow \Z\cup\{\infty\}$, which is zero on  $k$. A divisor $D$ of $F/k$ is a formal $\Z$-linear combination of the places of $F$. For a divisor $D=\sum_{P\in\pp_F}a_P\cdot P$, we set $v_P(D):=a_P$ and define the degree of $D$ (over $k$) by 
$$\deg_k D:=\sum_{P\in\pp_F}a_P\cdot\deg_k P.$$
For $z\in F^\times$ we define the principal divisor generated by $z$ by $(z):=\sum_{P\in\pp_F}v_P(z)\cdot P$. Denote by $Z_z:=\{P\in\pp_F\mid v_P(z)> 0 \}$ and $N_z:=\{P\in\pp_F\mid v_P(z)<0 \}$ the sets of zeros and poles of $z$, respectively. We call $(z)_0:=\sum_{P\in Z_z}v_P(z)\cdot P$ the zero divisor of $z$ and $(z)_\infty:=\sum_{P \in N_z}-v_P(z)\cdot P$ the pole divisor of $z$. The Riemann-Roch space of a divisor $D$ of $F$ is the finite dimensional $k$-vector space
$$\ll(D):=\{a\in F^\times\mid (a)\geq -D\}\cup \{0\}.$$ 
Instead of $\dim_k \ll(D)$, we write $\dim_k D$ for any divisor $D$ of $F$. Then, we may define the genus $g$ of $F$ as in the introduction.\\
Let $\oo_F:=\mathrm{Cl}(A,F)$ and $\oo_{F,\infty}:=\mathrm{Cl}(A_\infty,F)$ be the integral closures of $A$ and $A_\infty$ in $F$, respectively.\\ 
We realize an algebraic function field $F/k$ as the quotient field of the residue class ring $A[x]/f(t,x)A[x]$, where 
$$f(t,x)=x^n+a_1(t)x^{n-1}+\dots+a_n(t)\in A[x]$$ 
is irreducible, monic and separable in $x$. A polynomial $f$ satisfying these conditions is called a defining polynomial of $F/k$. Such a representation exists for every algebraic function field over a perfect constant field \cite[p. 128]{H.Stich}. We consider $\theta\in F$ with $f(t,\theta)=0$, so that $F$ can be expressed as $k(t,\theta)$. We call $A[\theta]$ the \emph{finite equation order} of $f$, and we define 
$$C_f:=\max\{\lceil\deg a_i(t)/i\rceil \mid 1\leq i\leq n\},\quad f_\infty(1/t,x):=t^{-nC_f}f(t,t^\cf x).$$ 
Thus, $f_\infty$ belongs to $k[1/t,x]\subset A_\infty[x]$ and $F$ can be represented as the quotient field of $A_\infty[x]/f_\infty(1/t,x) A_\infty[x]$. As $\theta_\infty:=\theta/t^\cf$ is integral over $A_\infty$, we may consider the \emph{infinite equation order} $A_\infty[\theta_\infty]$.

\begin{definition}
Let $M$ and $M'$ be two $A$- or $A_\infty$-modules of rank $n$ with bases $\{b_1,\dots,b_n\}$ and $\{b'_1,\dots,b'_n\}$, respectively. We set the index $[M:M']$ to be the class of $\det(T)$ in $K^*/k^*$, where $T\in\gl{n}{K}$ is a transition matrix.
\end{definition}

Note that this definition is independent of the choice of the bases of $M$ and $M'$. The values $v_\infty([M:M'])$ and $v_{p(t)}([M:M'])$, for an irreducible polynomial $p(t)\in A$, are defined as the valuation of  any representative of the class $[M:M']$.

Let $k$ be a finite field with $q$ elements. Our algorithm of the computation of the genus of a function field $F/k$ strongly depends on the Montes algorithm \cite{HN2}, \cite{HN}:\medskip

\noindent{\tt Montes-algorithm($f(t,x)$,\,$p(t)$)}\vskip 1mm

\noindent INPUT:

\noindent $-$ Defining polynomial $f(t,x)$ of a global function field $F/k$.

\noindent $-$ An irreducible polynomial $p(t)\in k[t]$.\medskip

\noindent OUTPUT:

\noindent $-$ Non-negative integer $\ind_{p(t)}:=v_{p(t)}([\oo_F:A[\theta]])$.\bigskip

The original version of the Montes algorithm produces a more comprehensive output, but we require only $\ind_{p(t)}$. Admitting fast multiplication, it is shown in \cite[Theorem 5.14]{BNS} that the Montes algorithm needs 
$$O\left((\deg p(t))^{1+\epsilon}(n^{2+\epsilon}+n^{1+\epsilon}\delta_{p(t)}\log(q^{\deg p(t)})+n^{1+\epsilon}\delta_{p(t)}^{2+\epsilon})\right)$$ 
operations in $k$ to terminate, where $n:=\deg f$ and $\delta_{p(t)}:=v_{p(t)}(\dsc(f))$. Also, if we set 
$$
\delta_\infty:=v_\infty(\dsc(f_\infty)),\quad\indi:=v_\infty([\oo_{F,\infty}:A_\infty[\theta_\infty]]),
$$
the routine \noindent{\tt Montes-algorithm($f_\infty(1/t,x)$,\,$1/t$)} determines the non-negative integer $\indi$, at a cost of $O(n^{2+\epsilon}+n^{1+\epsilon}\delta_{\infty}\log(q)+n^{1+\epsilon}\delta_{\infty}^{2+\epsilon})$ operations in $k$. 

\section{Lattices over $k(t)$}
Our aim is to describe the genus $g$ of a function field $F/k$ in terms of the indices $[\oo_F:A[\theta]]$, $[\oo_{F,\infty}:A_\infty[\theta_\infty]]$ and $[k_0:k]$. To this end, we use the language of lattices and their reduced bases. 
A more comprehensive consideration can be found in \cite{Len2} and \cite{Icke}.
\subsection{Lattices and normed spaces}

On $K=k(t)$ we consider the degree function $|~|:K\rightarrow  \{-\infty\}\cup \Z$, determined by $|x|:=-v_\infty(x)$ . Let $K_\infty=k((t^{-1}))$ be the completion of $K$ at the place $\infty$. The valuation $v_\infty$ extends in an obvious way to $\overline{K}_\infty$, and it determines a degree function on $\overline{K}_\infty$ as above: $|~|:=-v_\infty$. 
\begin{definition}\label{norm} Let $R$ be a subring of $K_\infty$, and let $X$ be a finitely generated $R$-module. A \emph{norm}, or  \emph{length function} on $X$ is a mapping
$$
\v{~}\colon X\lra \{-\infty\}\cup\R
$$ 
satisfying the following conditions:
\begin{enumerate}
\item $\Vert x+y\Vert\le \max\{\v{x},\Vert y\Vert\}$, for all $x,y\in X$, \item $\Vert ax\Vert=|a|+\v{x}$,
for all $a\in R$, $x\in X$,
\item $\v{x}=-\infty$ if and only if $x=0$,
\item $\dim_k\{x\in X\mid \v{x}\le r\}<\infty$, for each $r\in\R$. 
\end{enumerate}
\end{definition}

\noindent Clearly, $|~|:R\rightarrow \{-\infty\}\cup \R$, with $R\in \{K,K_\infty\}$, is a norm on $R$.
\begin{lemma}\label{ineq}
Let $R$ be a subring of $K_\infty$, $X$ a finitely generated $R$-module, and $\v{~}$ a norm on $X$. Then, for any $x_1,x_2\in X$ with $\v{x_1}\neq \v{x_2}$, it holds
$$
\v{x_1+x_2}=\max\{\v{x_1},\v{x_2}\}.
$$
\end{lemma}

\begin{proof}
Since $\v{x_1}\neq \v{x_2}$, we can assume $\v{x_1}>\v{x_2}$. Suppose that $\v{x_1+x_2}<\max\{\v{x_1},\v{x_2}\}=\v{x_1}$. We obtain 
$\v{x_1}=\v{(x_1+x_2)-x_2}\leq\max\{\v{x_1+x_2},\v{x_2}\}<\v{x_1},$
a contradiction.
\end{proof}

\begin{definition} A \emph{normed space over $K$} is a pair $(E,\v{~})$, where $E$
is a finite dimensional $K$-vector space and $\v{~}$ is a \emph{norm} on $E$.

A \emph{lattice over $A$} is a pair $(L,\v{~})$, where $L$ is a finitely generated
$A$-module, and $\v{~}$ is a \emph{norm} on $L$.

\end{definition}

Clearly, if $(L,\v{~})$ is a lattice, then $L\otimes_AK$ is a normed space, with
the norm function obtained by extending $\v{~}$ in an obvious way. The second property of
the norm function shows that $L$ has no $A$-torsion, so that $L$ is a free
$A$-module and it is embedded into the normed space $L\otimes_AK$. Conversely,
if $(E,\v{~})$ is a normed space, then any $A$-submodule of full rank is a
lattice with the norm function obtained by restricting $\v{~}$ to $L$.

\bigskip

\noindent{\bf Example 1. }The pair $(A,|~|)$ is a lattice in the normed space $(K,|~|)$, where $|~|$ is the ordinary degree function.
\medskip

\noindent{\bf Example 2. }Let $F/k$ be an algebraic function field. For each place $P$ of $F$ above the place $\infty$ of $K$, let $e(P|\infty)$ be the ramification index of $P$ over $\infty$. Define $w_P:=e(P|\infty)^{-1}v_P$ and:
$$w_\infty\colon F\lra \R\cup\{\infty\},\quad w_\infty(x):=\min_{P\mid\infty}\{w_P(x)\}. 
$$
Then, $(F,-w_\infty)$ is a normed space.
Actually, this is the normed space we are mostly interested in. 

\subsection{Reduced bases}
We fix throughout this section a normed space $(E,\v{~})$ over $K$, of dimension $n$. By a basis of $E$ we mean a $K$-basis. By a basis of a lattice $L\subset E$ we mean an $A$-basis.
Any basis of $L$ is in particular a basis of $E$. Conversely,
any basis $\bb=\{b_1,\dots,b_n\}$ of $E$, is a basis of the lattice $L:=\gen{\bb}_A$, the $A$-submodule generated by $\bb$. 

\begin{definition}
Let $\bb=\{b_1,\dots,b_m\}$ be a $K$-linearly independent family of $E$. We say that $\bb$ is \emph{reduced} if any of the following two equivalent conditions are satisfied:
\begin{enumerate}
\item $\v{a_1b_1+\cdots +a_mb_m}=\max\limits_{1\leq i\leq m}\{\v{a_ib_i}\}$, for all $a_1,\dots,a_m\in K$. 
\item $\v{a_1b_1+\cdots +a_mb_m}=\max\limits_{1\leq i\leq m}\{\v{a_ib_i}\}$, for all $a_1,\dots,a_m\in A$. 
\end{enumerate}
\end{definition}

\begin{theorem}\label{exist}
Every lattice admits a reduced basis.
\end{theorem}

\begin{proof}
In the literature, there are several proofs of this fact for particular normed spaces \cite{Len}, \cite{Sch}, \cite{Schoe}. It is not difficult to prove this for an abstract normed space \cite{Icke}, but we do not include the proof here because we do not need it for our purposes. 
\end{proof}

\subsection{Orthonormal basis and determinant}

\begin{definition}
Let $E$ be a normed space and $\bb$ a reduced basis of $E$. 
We say that $\bb$ is \emph{orthonormal} if $-1<\v{b}\leq 0$, for all $b\in\bb$. 
\end{definition}

\noindent Clearly, if $\bb$ is a reduced basis of $E$, then $\{t^{m_b}b\mid b\in\bb\}$ is an orthonormal basis, if we take $m_b=-\lceil\v{b}\rceil$. 

\noindent We now consider transition matrices between orthonormal bases. For two bases $\bb:=\{b_1,\dots,b_n\}$ and $\bb':=\{b'_1,\dots,b'_n\}$ in a normed space, a transition matrix from $\bb$ to $\bb'$ is defined to be a matrix $T\in \gl n{K}$ such that $(b'_1\dots b'_n)T=(b_1\dots b_n)$.

\begin{definition}
Let $m=m_1+\cdots+m_\kappa$ be a partition of a positive integer $m$ into a sum of positive integers. Let $T$ be a $m\times m$ matrix with entries in $A_\infty$. The partition of $m$ determines a decomposition of $T$ into blocks:
$$
T=(T_{ij}),\quad T_{ij}\in A_\infty^{m_i\times m_j},\ 1\le i,j\le \kappa.
$$  
\noindent The \emph{orthogonal group} $O(m_1,\dots,m_\kappa,A_\infty)$ is the set of all $T\in K^{m\times m}$ which satisfy the following two conditions:
\begin{enumerate}
\item $T_{ii}\in\gl{m_i}{A_\infty}$, for all $1\le i\le \kappa$.
\item $T_{ij}\in \m_\infty^{m_i\times m_j}$, for all $i> j$.
\end{enumerate}
\end{definition}

\begin{theorem}\label{transmat}

The \emph{orthogonal group} $O(m_1,\dots,m_\kappa,A_\infty)$ is a subgroup of \linebreak $\gl m {A_\infty}$. In particular, the determinant of a matrix in $O(m_1,\dots,m_\kappa,K)$ belongs to $U_\infty$.  
\end{theorem}
\begin{proof}
The image of $T\in O(m_1,\dots,m_\kappa,A_\infty)$ under the reduction homomorphism $A_\infty\rightarrow A_\infty/\m_\infty\cong k$ is an invertible matrix; hence $\det(T)\in U_\infty$.
\end{proof}

\begin{proposition}\label{transition} 
Let $\bb$ be an orthonormal basis of $E$, whose vectors are ordered by increasing length and let $m_1,\dots,m_\kappa$ be the multiplicities of the lengths of the vectors of $\bb$. Let $\bb'$ be a basis of $E$, whose vectors are ordered by increasing length. Then $\bb'$ is orthonormal if and only if the transition matrix from $\bb'$ to $\bb$ belongs to the orthogonal group $O(m_1,\dots,m_\kappa,A_\infty)$.  
\end{proposition}

\begin{proof}
A proof can be found in  \cite{Icke}. 
\end{proof}

\noindent Now we define the determinant of a lattice.
\begin {definition}
Let $\bb$ be a basis of a normed space $E$. We define $d(\bb)\in K^*/U_\infty$ to be the class modulo $U_\infty$ of the determinant of the transition matrix from $\bb$ to an orthonormal basis of $E$. We call $d(\bb)$ the \emph{determinant} of $\bb$. 
\end {definition}

This invariant is well-defined because the transition matrix between two orthonormal bases of $E$ has determinant in $U_\infty$ by Theorem \ref{transmat} and Proposition \ref{transition}.

\begin {definition}
Let $L$ be a lattice inside a normed space $E$. We define $d(L)\in K^*/U_\infty$ to be the determinant of any basis of $L$. We call $d(L)$ the \emph{determinant} of $L$.
\end {definition}

This invariant is well-defined because the transition matrix between two bases of $L$ has determinant in $k^*\subset U_\infty$. Note that $|d(\bb)|,|d(L)|\in\Z$ are well-defined.

\begin{lemma}\label{dL}
Let $E$ be a normed space,  $L\subset E$ a lattice and $\bb:=\{b_1,\dots,b_n\}$ a reduced basis of $L$. Then, the determinant of $L$ satisfies
$$
|d(L)|=\sum_{i=1}^n\lceil \v{b_i}\rceil.
$$
\end{lemma}

\begin{proof}
Since $\bb$ is a reduced basis, the set $\bb':=\{t^{m_1}b_1,\dots,t^{m_n}b_n\}$ with $m_i:=-\lceil \v{b_i}\rceil$, $1\leq i\leq n$, is an orthonormal basis of $E$ and $T:=\mathrm{diag}(t^{m_1},\dots,t^{m_n})^{-1}$ is a transition matrix from $\bb$ to $\bb'$. By definition, $|d(L)|$ satisfies
$$
|d(L)|=|\det(T)|=-\sum_{i=1}^nm_i=\sum_{i=1}^n\lceil \v{b_i}\rceil.
$$
\end{proof}

\section{Genus of function fields}
\subsection{Riemann-Roch theory and lattices}

Let $F/k$ be a function field of genus $g$ and denote $e(P|\infty)$ the ramification index of $P$ over $\infty$. We consider a divisor
$$
D+r(t)_\infty=\sum_{Q\in\mathbb{P}_0}\alpha_Q\cdot Q+\sum_{P\in\mathbb{P}_\infty}(\beta_P+re(P|\infty))\cdot P,
$$
with $r\in\Z$. The places $Q\in\mathbb{P}_0$ and $P\in\mathbb{P}_\infty$ are in 1:1 correspondence to prime ideals $\mathcal{Q}$ of $\oo_F$ and $\mathcal{P}$ of $\oo_{F,\infty}$, respectively. The Riemann-Roch space  of $D+r(t)_\infty$ satisfies
\begin{align*}
\mathcal{L}(D+r(t)_\infty) =\ii_0\cap\ii_\infty
\end{align*}
with fractional ideals $\ii_0:=\prod_{Q\in\mathbb{P}_0}\mathcal{Q}^{-\alpha_Q}$ and $\ii_\infty:=t^{-r}\cdot\prod_{P\in\mathbb{P}_\infty}\mathcal{P}^{-\beta_P}$ of $\oo_F$ and $\oo_{F,\infty}$, respectively. We consider the norm on $F$:
\begin{align*}
\v{~}:F\rightarrow \{-\infty\}\cup\Q,\ \v{z}=-\min_{P\in\mathbb{P}_\infty}\left\{\frac{v_P(z)+v_P(D)}{e(P|\infty)}\right\}.
\end{align*}
Thus, $(F,\v{~})$ becomes a normed space. The fractional ideal $\ii_0$, equipped with the norm $\v{\ }$, is a lattice in $(F,\v{\ })$. Clearly, any divisor $D$ induces a norm $\v{~}_D$ and a normed space $(F,\v{~}_D)$. As our consideration is relative to a fixed divisor $D$, we write $\v{~}$ instead of $\v{~}_D$. Note that $\ii_0$ does not depend on $r$.

\begin{theorem}[{\cite{Schoe}[Satz III.17]}]
Let $\bb:=\{b_1,\dots,b_n\}$ be a reduced basis of $\ii_0$. Then,
$$
\{b_it^{j_i}\mid 1\leq i\leq n,\  0\leq j_i\leq- \lceil \v{b_i}\rceil +r\}
$$
is a $k$-basis of $\mathcal{L}(D+r(t)_\infty)$.
\end{theorem}

\begin{proof}
Let $z=\sum_{i=1}^n\lambda_ib_i$, with $\lambda_i$ in $A$, be an arbitrary element of $\ii_0$. The element $z$ belongs to $\mathcal{L}(D+r(t)_\infty)=\ii_0\cap \ii_\infty$ if and only if $v_P(z)\geq -v_P(D)-re(P|\infty)$, for all $P\in\mathbb{P}_\infty$. This condition can be expressed as
\begin{align*}
 \min_{P\in\mathbb{P}_\infty}\left\{\frac{v_P(z)+v_P(D)}{e(P|\infty)}\right\}\geq -r
\Longleftrightarrow  \v{z}\leq r\Longleftrightarrow& \max_{i=1}^n\{\v{\lambda_ib_i}\}\leq r,
\end{align*}
or equivalently, $|\lambda_i| \leq -\v{b_i}+r$, for all $i\in\{1,\dots,n\}$.
Since the coefficients $\lambda_i$ are polynomials, we obtain $z\in\mathcal{L}(D+r(t)_\infty)$ if and only if $|\lambda_i|=\deg(\lambda_i)\leq \lfloor -\v{b_i}\rfloor+r=- \lceil \v{b_i}\rceil+r$, for all $1\leq i\leq n$.
\end{proof}

\begin{corollary}\label{dim}
Let $\bb:=\{b_1,\dots,b_n\}$ be a reduced basis of $\ii_0$. Then,
$$
\dim_k(D+r(t)_\infty)=\sum_{\lceil\v{b_i}\rceil\leq r}(-\lceil\v{b_i}\rceil +r+1).
$$
\end{corollary}

\begin{corollary}[{\cite{F.H.}[Corollary 5.5]}]\label{flo}
Let $D$ be a divisor of $F$, and $\ii_0$ and $\ii_\infty$ fractional ideals of $\oo_F$ and $\oo_{F,\infty}$, respectively, such that $\ll(D)=\ii_0 \cap \ii_\infty$ . Then, it holds
\begin{align*}
-|d(\ii_0)|=\deg_k(D)+[k_0:k](1-g)-n,
\end{align*}
where $k_0$ is the full constant field of $F/k$.
\end{corollary}

\begin{proof}
Let $\{b_1,\dots,b_n\}$ be a reduced basis of $\ii_0$. For a sufficiently large $r\in \Z$, Corollary \ref{dim} shows that $\dim_k(D+r(t)_\infty) =(\sum_{i=1}^n-\lceil\v{b_i}\rceil )+nr+n$. Also, for large $r$ we obtain by the Riemann-Roch theorem
\begin{align}\label{GenusFlo}
\dim_k(D+r(t)_\infty)&=\deg_k(D+r(t)_\infty)+[k_0:k](1-g)\\
&=\deg_k D+rn+[k_0:k](1-g).\notag
\end{align}
So, finally we have $\deg_kD+[k_0:k](1-g)-n=-\sum_{i=1}^n\lceil \v{b_i}\rceil=-|d(\ii_0)|$, where the last equality follows from Lemma \ref{dL}.
\end{proof}

\begin{theorem}\label{preformular}
Let $F/k$ be a function field with defining polynomial $f$ and let $\theta\in F$ be a root of $f$. Let $D$ be a divisor with $\mathcal{L}(D)=\ii_0\cap \ii_\infty$, where $\ii_0$ and $\ii_\infty$ are fractional ideals of $\oo_F$ and $\oo_{F,\infty}$, respectively. Then,
\begin{align*}
|d(\ii_0)|=-|[\ii_0:A[\theta]]|+|[\ii_\infty:A_\infty[\theta_\infty]]|+\cf n(n-1)/2.
\end{align*}
\end{theorem}
\begin{proof}
Let $\bb:=\{b_1,\dots,b_n\}$ be any basis of $\ii_0$ and $\bb':=\{b'_1,\dots,b'_n\}$ an orthonormal basis of $\ii_\infty$. We consider $M,M'\in K^{n\times n}$ with $(1\  \theta \dots\theta^{n-1})M=(b_1\dots b_n)$ and $(1\  \theta \dots\theta^{n-1})M'=(b'_1 \dots b'_n)$. Then, $T:=M'^{-1}M$ is a transition matrix from $\bb$ to $\bb'$ and by definition, 
\begin{align}\label{formula}
|d(\ii_0)|=|\det(T)|=|\det(M'^{-1})|+|\det(M)|=|\det(M'^{-1})|-|[\ii_0:A[\theta]]|.
\end{align}

\noindent Let $N:=\mathrm{diag}(1,t,\dots t^{-\cf(n-1)})$. Clearly, 
\begin{align*}
(1\ \theta_\infty\dots \theta_\infty^{n-1}) =(1\  \theta \dots\theta^{n-1})N=(b'_1 \dots b'_n)M'^{-1}N.
\end{align*}
Hence, $|[\ii_\infty:A_\infty[\theta_\infty]]|=|\det(M'^{-1}N)|=|\det(M'^{-1})|-\cf n(n-1)/2$, and therefore
$$
|\det(M'^{-1})|=|[\ii_\infty:A_\infty[\theta_\infty]|+\cf n(n-1)/2.
$$
Together with (\ref{formula}), we obtain the claimed formula for $|d(\ii_0)|$. 
\end{proof}

\subsection{Computation of the genus}

We apply Corollary \ref{flo} to the zero divisor $D:=(0)$. Then, $\ii_0$ becomes $\oo_F$ and $\ii_\infty=\oo_{F,\infty}$. Therefore,
$$
g=\frac{[k_0:k]-n+|d(\ii_0)|}{[k_0:k]}.
$$
By Theorem \ref{preformular} we obtain the following result.
\begin{corollary} 
For a function field $F/k$ with defining equation $f(t,\theta)=0$, the genus may be computed as:
\begin{align*}
g=\frac{[k_0:k]-n-|[\oo_F:A[\theta]]|+|[\oo_{F,\infty}:A_\infty[\theta_\infty]]|+\cf n(n-1)/2}{[k_0:k]}.
\end{align*}
If $[k_0:k]=1$, we obtain
$$g=1-n-|[\oo_F:A[\theta]]|+|[\oo_{F,\infty}:A_\infty[\theta_\infty]]|+\cf n(n-1)/2.$$
\end{corollary}
\bigskip
A conventional way to compute the genus $g$ of a function field $F/k$ proceeds as follows: Consider the divisor $D:=(r(t)_\infty)$ and the Riemann-Roch space $\ll(D)=\oo_F\cap t^{-r}\cdot\oo_{F,\infty}$. Determine the $\lceil\v{~}_D\rceil$-values of a $\v{~}_D$-reduced basis $\{b_1,\dots,b_n\}$ of $\oo_F$. For large $r$ (i.e. $r\geq 2g-1$), Corollary \ref{dim} and (\ref{GenusFlo}) show that
\begin{align}\label{genus2}
\sum_{\lceil \v{b_i}_D\rceil\leq r}(-\lceil \v{b_i}_D\rceil+r +1)=\dim_k D=rn +[k:k_0](1-g).
\end{align}
Since $[k:k_0]=\dim_k\ll(0)=\sum_{\lceil \v{b_i}_D\rceil\leq r}(-\lceil \v{b_i}_D\rceil +1)$, the genus $g$ can easily be deduced from (\ref{genus2}).\medskip

If the constant field $k$ is algebraically closed in the function field $F$ (e.g global function fields), the computation of the genus $g$ of $F$ can be reduced to the computation of the degree of the two indices $[\oo_F:A[\theta]]$ and $[\oo_{F,\infty}:A_\infty[\theta_\infty]]$. The valuations
$$
\ind_{p(t)}:=v_{p(t)}([\oo_F:A[\theta]]),\quad \indi:=v_\infty([\oo_{F,\infty}:A_\infty[\theta_\infty]])
$$
of these two indices are computed by the Montes algorithm \cite{HN2}, \cite{HN}. Therefore, we have the following method to determine $g$.

\begin{algorithm}\label{Algo2}
\caption{: Genus computation of global function fields}
\label{Algo2}
\begin{algorithmic}[1]

\REQUIRE A global function field $F/k$ with defining polynomial $f$ of degree $n$.
\ENSURE Genus $g$ of $F$.\\[0.5cm]

\STATE $f_\infty \leftarrow t^{-\cf n}f(t,t^\cf x)$
\STATE FiniteIndex$\ \leftarrow 0$
\STATE Factorize $\dsc(f)$
\FOR{all irreducible polynomials $p(t)$ with $v_{p(t)}(\dsc(f))\geq 2$}\label{loop}
\STATE$\ind_{p(t)}\leftarrow$ \noindent{\tt Montes-algorithm($f$,\,$p(t)$)}
\STATE FiniteIndex $\leftarrow$ FiniteIndex $+|p(t)|\cdot \ind_{p(t)}$
\ENDFOR
\STATE $\ind_\infty\leftarrow$\noindent{\tt Montes-algorithm($f_\infty$,\,$1/t$)}
\STATE \textbf{return} $1-n-$FiniteIndex$-\ind_\infty+\cf n(n-1)/2$

\end{algorithmic}
\end{algorithm}

\begin{theorem}\label{complexi}
Let $F/k$ be a function field over the finite field $k$ with $q$ elements and with defining polynomial $f$ of degree $n$. Then, Algorithm \ref{Algo2} needs at most 
$$
O(n^{5+\epsilon}\cf^{2+\epsilon}\log(q))
$$
operations in $k$ to determine the genus of $F$.
\end{theorem}

\begin{lemma}\label{boundy}
Let $F/k$ be function field of genus $g$ with defining polynomial $f$ of degree $n$. Then, $\delta:=|\dsc(f)|$ and $\delta_\infty:=v_\infty(\dsc(f_\infty))$ satisfy
$$
\delta\leq \delta+\delta_\infty=\cf n(n-1)=O(n^2\cf).
$$
\end{lemma}

\begin{proof}
The discriminant of $f_\infty$ belongs to $k[1/t]$. Hence, $\delta_\infty=v_\infty(\dsc(f_\infty))\geq 0$ and $\delta\leq \delta+\delta_\infty$. As $f_\infty=t^{-nC_f}f(t,t^\cf x)$, the discriminant of $f_\infty$ satisfies  \cite[p. 13]{EN} $$
\dsc(f_\infty)=t^{-nC_f(2n-2)}t^{\cf(n^2-n)}\dsc(f)=t^{-C_fn(n-1)}\dsc(f).
$$
Therefore $v_\infty(\dsc(f_\infty))=\cf n(n-1)-|\dsc(f)|$.
\end{proof}

\begin{proof}[Proof of theorem \ref{complexi}]
Initially, Algorithm \ref{Algo2} computes $\dsc(f)$ and factorizes it. Since $\dsc(f)=\res(f,f')$, the cost of the computation of $\dsc(f)$ is equal to the cost of computing the determinant of the Sylvester matrix $M$ of $f$ and $f'$. In \cite{MS} it is shown that the cost is $O(d^2n^3)$ field operations, where $d$ denotes the maximal degree of the entries in $M$. Since $d=O(n\cf)$, the computation of $\dsc(f)$ needs at most $O(n^5\cf^2)$ operations in $k$.\\
The factorization of $\dsc(f)$ can be estimated by $O((n^2\cf )^{2+\epsilon}+(n^2\cf )^{1+\epsilon}\log(q))=O( n^{4+\epsilon}\cf^{2+\epsilon}+n^{2+\epsilon}\cf^{1+\epsilon} \log(q))$ operations in $k$, since $\deg \dsc(f)=|\dsc(f)|=O(n^2\cf)$ (c.f. Lemma \ref{boundy}). \\
Considering \cite[theorem 5.14]{BNS} the cost of one call of  \noindent{\tt Montes-algorithm($f$,\,$p(t)$)} and \noindent{\tt Montes-algorithm($f_\infty$,\,$1/t$)}  is equal to 
$$
O((\deg p(t))^{1+\epsilon}(n^{2+\epsilon}+n^{1+\epsilon}\delta_{p(t)} \log(q^{\deg p(t)})+n^{1+\epsilon}\delta_{p(t)}^{2+\epsilon}))
$$
and $O\left(n^{2+\epsilon}+n^{1+\epsilon}\delta_{\infty} \log(q)+n^{1+\epsilon}\delta_{\infty}^{2+\epsilon}\right)$
operations in $k$, respectively. The worst case is that we have to call the Montes-algorithm for all prime divisors $p(t)$ of $\dsc(f)$. Therefore the cost of the for-loop is
\begin{align*}
\sum_{p(t)|\dsc(f)}&O((\deg p(t))^{1+\epsilon}(n^{2+\epsilon}+n^{1+\epsilon}\delta_{p(t)} \log(q^{\deg p(t))})+n^{1+\epsilon}\delta_{p(t)}^{2+\epsilon}))\\
=\ &O(n^{2+\epsilon}\delta^{1+\epsilon}+n^{1+\epsilon}\delta^{2+\epsilon}\log(q)+n^{1+\epsilon}\delta^{2+\epsilon})
\end{align*}
operations in $k$. Adding the cost for applying \noindent{\tt Montes-algorithm($f_\infty$,\,$1/t$)}, we obtain for the computation of the degree of the two indices
\begin{align*}
&O(n^{2+\epsilon}\delta^{1+\epsilon}+n^{1+\epsilon}(\delta^{2+\epsilon}+\delta_\infty)\log(q)+n^{1+\epsilon}(\delta^{2+\epsilon}+\delta_\infty^{2+\epsilon}))\\
=&O(n^{4+\epsilon}\cf^{1+\epsilon}+n^{5+\epsilon}\cf^{2+\epsilon}\log(q)+n^{5+\epsilon}\cf^{2+\epsilon})\\
=&O(n^{5+\epsilon}\cf^{2+\epsilon}\log(q)),
\end{align*}
field operations, where the second equality follows from Lemma \ref{boundy}. Clearly, the cost of the computation and  factorization of $\dsc(f)$ is dominated by\linebreak $O(n^{5+\epsilon}\cf^{2+\epsilon}\log(q))$.
\end{proof}

\begin{remark}
In \cite[Lemma 1.29]{JDB} is shown that for $g\geq 1$ the defining polynomial $f$ of a global function field $F/k$ can be chosen such that $\cf=O(g^2)$ holds. In that case Algorithm \ref{Algo2} needs at most $O(n^{5+\epsilon}g^{4+\epsilon}\log(q))$ operations in $k$ to determine the genus of $F$.
\end{remark}

\section{Experimental results}

We have implemented Algorithm \ref{Algo2} in Magma \cite{BCP}. The code may be downloaded from  \href{http://montesproject.blogspot.com}{http://montesproject.blogspot.com}. In this section we compare the runtime of our algorithm with that of the algorithm of Magma. 
The computations have been done in a Linux server, with two Intel Quad Core processors, running at 3.0 Ghz, with 32 Gb of RAM memory. Times are expressed in seconds. If an algorithm did not terminate after $24$ hours we write $"-"$ instead. For each example we present the characteristic data of the function field $F/k$ and its defining polynomial $f$ and the time, which needed Algorithm \ref{Algo2} and that of Magma to determine the genus. Additionally, we give the number of seconds of the initial computation (I.C.) in Algorithm \ref{Algo2}; that is, the time that costs the computation of $\dsc(f)$ and its factorization. We will see that in most of the cases the runtime of the initial computation dominates the runtime of Algorithm \ref{Algo2}. In the column Algo \ref{Algo2} we display the total running time of Algorithm \ref{Algo2}, including the initial computation.
 \medskip\\
For the first examples we use families of global function fields, which cover all the computational difficulties of the Montes algorithm \cite{GNP}. Later, we use randomly chosen global function fields.
 \medskip\\

We consider in all examples the function field $F/k$ of genus $g$, with defining polynomial $f(t,x)\in k[t,x]$.
\subsubsection*{Example 1}
Let $f=(x+p(t)^r+\dots+1)^n+p(t)^k\in \mathbb{F}_{37}[t,x]$, where $p(t)\in A$ is irreducible and $k,r$ are non-negative integers.\\

\begin{tabular}{|c|c|c|c|c||c|c||c|c|c|}\hline $g$&$p(t)$&$n$&$k$&$r$& $\delta$&$\delta_\infty$& I.C.&Algo \ref{Algo2} & Magma \\\hline 
$0$&$t$& $5$&$7$&$10$ & $28$&$172$&0.0 &0.02  & 0.39   \\\hline
$22$&$t^3+2$& $23$&$30$&$10$ & $1980$&$13200$&8.08 &8.31  & 66289.34   \\\hline
$0$&$t+1$& $77$&$163$&$20$ & $12388$&$104625$&265.4&267.91  & $-$   \\\hline
 \end{tabular}

\subsubsection*{Example 2}
Let $f=(\prod_{\alpha\in\mathbb{F}_3}(x+t\alpha)^m+tp(t)^k)^m+tp(t)^{3mk}\in\mathbb{F}_{3}[t,x]$, where $p(t)=t^2+1$ and $m,k$ are  non-negative integers.\\

\begin{tabular}{|c|c|c|c||c|c||c|c|c|}\hline $g$&$\deg(f)$&$k$&$m$& $\delta$&$\delta_\infty$&I.C.&Algo \ref{Algo2} & Magma \\\hline 
$50$& $12$&$2$ &$2$& $264$&$132$&0.01  &0.05& 0.82  \\\hline
$528$& $48$&$5$ &$4$ & $5640$&$1128$&1.13  &3.96 & 1322.08 \\\hline
$1136$& $75$&$7$ &$5$ & $15510$&$1140$&9.0 &37.9  & 15961.82  \\\hline
$1198$& $147$&$1$ &$7$ & $7854$&$13608$&2.60 &80.3  & $-$   \\\hline
 \end{tabular}

\subsubsection*{Example 3}
Let $f=(x^2-2x+4)^3+p^k\in\mathbb{F}_{q}[t,x]$, where $p(t)\in A$ is a prime polynomial and $k$ a non-negative integer.\\

\begin{tabular}{|c|c|c|c||c|c||c|c|c|}\hline $g$&$q$&$p(t)$&$k$& $\delta$&$\delta_\infty$&I.C.&Algo \ref{Algo2} & Magma \\\hline 
$0$&$7$& $t+2$&$7$ & $35$&$25$&0.00  &0.01 & 0.04  \\\hline
$60$&$7$& $t+2$&$122$ & $610$&$20$&0.02  &0.03 & 0.36  \\\hline
$450$&$101$& $t+1$&$901$ & $4505$&$25$&1.51  &1.57 & 15.76  \\\hline
$3512$&$73$& $t^2+1$&$3511$ & $35510$&$20$&186.56  &189.1 & 1238.77   \\\hline

 \end{tabular}

\subsubsection*{Example 4}
Let $f=((x^6+4p(t)x^3+3p(t)^2x^2+4p(t)^2)^2+p(t)^6)^3+p(t)^k\in \mathbb{F}_{q}[t,x]$ of degree $36$, where $p(t)\in A$ is an irreducible polynomial and $k$ a non-negative integer.\\

\begin{tabular}{|c|c|c|c||c|c||c|c|c|}\hline $g$&$q$&$p(t)$&$k$& $\delta$&$\delta_\infty$&I.C. &Algo \ref{Algo2} & Magma \\\hline 
$85$&$13$& $t^2+1$&$11$ & $924$&$336$&0.05  &0.19 & 122.6  \\\hline
$519$&$101$& $t+17$&$112$ & $3920$&$1120$&0.83  &4.59 & 1052.82 \\\hline
$3379$&$53$& $t^2+2$&$323$  & $22610$&$70$&23.8  &238.5 & 4617.74  \\\hline

 \end{tabular}

\subsubsection*{Example 5}
Let $f=(x^{l-1}+\dots+x +1)^m+t^k\in\mathbb{F}_{q}[t,x]$, where $m,l,k$ are non-negative integers.\\

\begin{tabular}{|c|c|c|c|c|c||c|c||c|c|c|}\hline $g$&$q$&$\deg f$&$m$&$l$&$k$& $\delta$&$\delta_\infty$&I.C.&Algo \ref{Algo2} & Magma \\\hline 
$6$&$101$& $8$&$4$&$3$&$13$ & $91$&$21$&0.00  &0.01 & 0.10  \\\hline
$0$&$13$& $42$&$7$&$7$&$13$  & $533$&$1189$&0.01  &0.17 & 292.9  \\\hline
$2$&$3$& $260$&$13$&$21$&$2$  & $518$&$66822$&0.02  &0.82 & $-$   \\\hline
$36$&$13$& $420$&$21$&$21$&$5$  & $2095$& 173885 &1.45  &3.81 & $-$  \\\hline

 \end{tabular}

\subsubsection*{Example 6}
For $1\leq l\leq 6$ we take the family of polynomials $f_l\in \F_{13}[t,x]$ with: \\

\begin{tabular}{|l|}
\hline
$f_{ 1}(t,x)=x^2+t$\\
\hline
$f_{ 2}(t,x)=f_{1}(x)^{2}+(t-1)t^{3}x$\\
\hline
$f_{ 3}(t,x)=f_{2}(x)^{3}+t^{11}$\\
\hline
$f_{ 4}(t,x)=f_{3}(x)^{3}+t^{29}x f_{2}(x)$\\
\hline
$f_{ 5}(t,x)=f_{4}(x)^{2}+(t-1)t^{42}xf_{1}(x)f_{3}(x)^2$ \\
\hline
$f_{ 6}(t,x)=f_{5}(x)^{2}+t^{88}xf_{3}(x)f_{4}(x)$\\
\hline

\end{tabular}\\\medskip

\begin{tabular}{|c|c|c||c|c||c|c|c|}\hline $l$&$g$&$\deg f_l$& $\delta$&$\delta_\infty$&I.C.&Algo \ref{Algo2} & Magma \\\hline 
$1$&$0$& $2$& $1$&$1$&0.0  &0.0 & 0.0  \\\hline
$2$&$3$& $4$& $16$&$8$&0.0  &0.01 & 0.01  \\\hline
$3$&$9$& $12$& $136$&$128$&0.01  &0.02 & 1.66  \\\hline
$4$&$40$& $36$& $1223$&$1297$&0.09  &0.12 & 707.06   \\\hline
$5$&$133$& $72$& $4964$&$4671$&1.37  &1.61 &60125.24  \\\hline
$6$&$329$& $144$& $19618$&$21566$&22.06  &24.67 & $-$  \\\hline

 \end{tabular}

\subsubsection*{Example 7}
We consider the function field $F/\mathbb{F}_{q}$ of genus $g=140$ with the defining polynomial $x^{41} -(t^2+1)(x^2-1)- (t^8+2t^6 +1)x$.\\

\begin{tabular}{|c|c|c||c|c|c|}\hline $q$&$\delta$&$\delta_\infty$&I.C.&Algo \ref{Algo2} & Magma \\\hline 
3&328&1312& 0.01  & 0.05 & 64.23 \\\hline
97&328&1312& 0.03  & 0.03 & 169.02 \\\hline 
10007&328&1312& 0.08  & 0.10 & 171.98 \\\hline 
 \end{tabular}

\subsubsection*{Example 8}
We consider the function field $F/\mathbb{F}_{q}$ of genus $g=213$ with the defining polynomial $x^{62}+(t+1)x^{12}+t^8+1$.\\

\begin{tabular}{|c|c|c||c|c|c|}\hline $q$& $\delta$&$\delta_\infty$&I.C.&Algo \ref{Algo2} & Magma \\\hline 
 7& 488&3294&0.04 &58.01 & 765.12 \\\hline
 113& 488&3294&0.05 &148.89 & 2011.94 \\\hline
 1013& 488&3294&0.04 &163.41 & 2017.90 \\\hline
  \end{tabular}

\subsubsection*{Example 9}
We consider the function field $F/\mathbb{F}_{q}$ of genus $g=325$ with the defining polynomial $x^{94}+(t+1)x^{12}+t^8+1$.\\

\begin{tabular}{|c|c|c||c|c|c|}\hline  $q$&$\delta$&$\delta_\infty$&I.C.&Algo \ref{Algo2} & Magma \\\hline 
 7&744&7998&0.06  & 152.90 & 5990.76  \\\hline
 103&744&7998&0.07  & 601.51 & 23528.40  \\\hline
  1009&744&7998& 0.10  &  2707.93 &    24305.16 \\\hline 
  \end{tabular}

\subsubsection*{Example 10}
We consider the function field $F/\mathbb{F}_{q}$ of genus $g$ with the defining polynomial $ x^{40}+(t+1)x^{23}+t^9x+(t+1)x^{13}+(t^5-3t^2)x^7+t^{62}x^3+t+1$.\\

\begin{tabular}{|c|c|c|c||c|c|c|}\hline  $g$&$q$&$\delta$&$\delta_\infty$&I.C.&Algo \ref{Algo2} & Magma \\\hline 
1220&5&2482&638&1.07  &1.13 & 92.5  \\\hline 
1220&125&2482&638&1.05  & 1.07 & 98.02  \\\hline 
1221&3137&2482&638&15.79  & 15.81 & 212.37  \\\hline 
\end{tabular}

\subsubsection*{Example 11}
We consider the function field $F/\mathbb{F}_{q}$ of genus $g$ with the defining polynomial $x^{68}+(t+1)^4x^{23}+(t^3+5)^9x+(t+1)x^{13}+(t^5-3t^2)x^7+t^{62}x^3+t+1$.\\

\begin{tabular}{|c|c|c|c||c|c|c|}\hline $g$&$q$& $\delta$&$\delta_\infty$&I.C.&Algo \ref{Algo2} & Magma \\\hline 
2082&5&1836&2720& 0.59  & 0.62 & 514.14  \\\hline 
2082&125&1836&2720& 0.64  & 0.67 & 598.40 \\\hline 
2083&3137&4235&321& 45.44 & 45.48 & 1662.31  \\\hline 
\end{tabular}

\subsubsection*{Example 12}
We consider the function field $F/\mathbb{F}_{q}$ of genus $g=3669$ with the defining polynomial $x^{120}+(t+1)^4x^{23}+(t^3+5)^9x+(t+1)x^{13}+(t^5-3*t^2)x^7+t^{62}x^3+t+1$.\\

\begin{tabular}{|c|c|c||c|c|c|}\hline $q$& $\delta$&$\delta_\infty$& I.C.&Algo \ref{Algo2} & Magma \\\hline 
5& 7459&6821& 16.56  & 16.59 & 15415.75  \\\hline 
97& 7459&6821& 70.23  & 70.28 & 21610.17  \\\hline 
529& 7459&6821& 182.88  & 182.99 & 16172.12  \\\hline 

\end{tabular}

\subsubsection*{Example 13}
We consider the function field $F/\mathbb{F}_{q}$ of genus $g=15154$ with the defining polynomial $x^{4330} -(t^2+1)(x^2-1)- (t^8+2t^6 +1)x$.\\

\begin{minipage}{\linewidth}

\renewcommand{\footnoterule}{}
\renewcommand{\thefootnote}{\alph{footnote}}

\begin{tabular}{|c|c|c||c|c|c|}
\hline $q$&$\delta$&$\delta_\infty$&I.C.&Algo \ref{Algo2} & Magma \\\hline 
3& 8674&$18744570$&85.99  & 95.82 & $-$\footnotemark[1] \\\hline 
37& 8674&$18744570$&760.92  & 802.87 & $-$\footnotemark[1] \\\hline 
\end{tabular}
\footnotetext[1]{All virtual memory has been exhausted, so Magma cannot perform this statement.}

\end{minipage}

\end{document}